\documentclass[12pt]{amsart}
\usepackage[headings]{fullpage}

\usepackage{amsmath,amssymb,amsthm, bm}
\usepackage{graphicx}
\usepackage{enumerate}
\usepackage{url}
\usepackage{slashed}
\usepackage{bm}
\usepackage[german,english]{babel}

\usepackage[bookmarks=true,%
    colorlinks=true,%
    linkcolor=blue,%
    citecolor=blue,%
    filecolor=blue,%
    menucolor=blue,%
    urlcolor=blue,%
    breaklinks=true]{hyperref}

\usepackage[all]{xy}
\usepackage{verbatim}

\newtheorem{thm}{Theorem}[]
\newtheorem*{thm*}{Theorem}
\newtheorem{lem}[thm]{Lemma}

\newtheorem{cor}[thm]{Corollary}

\newtheorem{rem}[thm]{Remark}

\newcommand{\param}{{\mathchoice{\mkern1mu\mbox{\raise2.2pt\hbox{$
\centerdot$}}
\mkern1mu}{\mkern1mu\mbox{\raise2.2pt\hbox{$\centerdot$}}\mkern1mu}{
\mkern1.5mu\centerdot\mkern1.5mu}{\mkern1.5mu\centerdot\mkern1.5mu}}}

\renewcommand{\setminus}{{\smallsetminus}}

\graphicspath{{figures/}}

\begin{document}
\title[Simple curves in covering spaces]{Finding simple curves in surface covers is undecidable}
\author{Ingrid Irmer}
\address{Mathematics Department\\
Technion, Israel Institute of Technology\\
Haifa, 32000\\
Israel 
}
\email{ingridi@technion.ac.il}
\today

\begin{abstract}
It is shown that various questions about the existence of simple closed curves in normal subgroups of surface groups are undecidable.
\end{abstract}

\maketitle

{\footnotesize
\tableofcontents
}

\section{Introduction}
\label{intro}
Let $S$ be a (not necessarily closed) pointed surface with genus at least 2, and $\tilde{S}$ a regular covering space of $S$. Given a word in $\pi_{1}(S)$ there are numerous algorithms that will determine whether this word represents a simple curve on $S$, for example \cite{BirmanSeries}, \cite{Chillingworth}, \cite{R} and \cite{Z}. However, if we want to decide whether (a) a given word representing a simple curve is contained in $\pi_{1}(\tilde{S})$, or (b) whether there exists a word in $\pi_{1}(\tilde{S})$ representing a simple curve in $S$, it turns out that these two questions are undecidable. The former is a consequence of undecidability of the word problem for finitely presented groups, and follows easily from Corollary \ref{corollary}, and undecidability of the latter is proven in this paper.

\begin{thm}
	\label{maintheorem}
There does not exist an algorithm that determines whether a normal subgroup $\pi_{1}(\tilde{S})\subset \pi_{1}(S)$ contains a word representing a homotopically nontrivial, simple curve on $S$.	
\end{thm}

When $\pi_{1}(\tilde{S})$ does not contain a word representing a simple curve in $\pi_{1}(S)$, we will say that the cover $\tilde{S}$ does not have any simple curves. Note that ``simpleness'' here refers to $S$, not to $\tilde{S}$.

\begin{rem}
	Most questions about groups only depend on the groups up to isomorphism. However, Theorem \ref{maintheorem} is not just about the group $\pi_{1}(\tilde{S})$, but depends on how $\pi_{1}(\tilde{S})$ is embedded in the larger group $\pi_{1}(S)$.
\end{rem}

Theorem \ref{maintheorem} arose out of an attempt to devise an algorithm for testing the simple loop conjecture. Let $f: S \rightarrow M$ be a 2-sided immersion of a closed surface $S$ in a 3-manifold $M$, and denote by $f_{\sharp}:\pi_{1}(S) \rightarrow \pi_{1}(M)$ the induced map on the fundamental group. The simple loop conjecture states that if the normal subgroup $ker(f_{\sharp})$ is nontrivial, then there exists an essential simple closed curve in $ker(f_{\sharp})$. The simple loop conjecture has been proven in many special cases, for example Seifert fibered 3-manifolds, \cite{Hass}, graph 3-manifolds, \cite{RubinsteinWang} and 3-manifolds with geometric structures modelled on Sol, \cite{Zemke}.\\

There is a sense in which covering spaces not containing simple curves are generic. One reason for this is that any cover factoring through another cover that does not contain simple curves, will also not contain simple curves; some examples will be given in Section \ref{section2}. The simple loop conjecture is therefore making a very strong claim about maps of surfaces to 3-manifolds, and this paper raises doubts about whether the conjecture is verifiable algorithmically.\\

\textbf{Outline of paper.} Section \ref{section1} will revise some known properties of covering spaces, and Section \ref{section2} will construct examples of covering spaces that do not have simple curves. The techniques used in these examples will be used in Section \ref{section3} to give a proof of Theorem \ref{maintheorem}, by modifying Rabin's construction.

\section*{Acknowlegments}
This paper benefitted from helpful discussions with Andrew Putman (who suggested this problem in another context), Hyam Rubinstein and the group theorists and their visitors at the Technion, in particular Giles Gardam, Tobias Hartnick, Anton Khoroshkin and Michah Sageev. I would like to thank the Technion for its hospitality while this paper was written, and my host Yoav Moriah.

\section{Notation and assumptions}
Since the 2-sphere has no nontrivial covers, and for the torus, homological methods give a simple algorithm for finding all simple curves on a covering space, it will be assumed that the surface $S$ has genus $g$ at least 2.\\

When $S$ is orientable and has empty boundary, the usual presentation $\langle a_{1}, \ldots, a_{2}, b_{1}, \ldots, b_{g}\mid \Pi_{i=1}^{i=g} [a_{i}, b_{i}]\rangle$ is assumed, where the curves $\{a_{i}, b_{i}\}$ are shown in Figure \ref{basis} in the case $g=2$.\\

\begin{figure}
	\centering
	\includegraphics{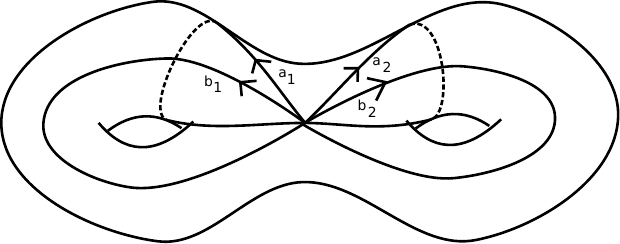}
	\caption{A set of simple curves representing a basis for $\pi_{1}(S)$.}
	\label{basis}
\end{figure}

A \emph{curve} on a surface is the image of a map of $S^1$ into the surface. In this convention, a curve is necessarily closed and connected. A curve $\gamma$ is said to be \textit{simple} if $\gamma$ is freely homotopic to an embedding of $S^1$. 

\section{Elementary properties of covers}
\label{section1}
This section recalls some properties of covering spaces. For this section and the next, it will be assumed that $S$ is orientable.\\

Given a regular cover $\tilde{S}$ of $S$, there is the following exact sequence, where $D$ denotes the deck transformation group of the cover.

\begin{equation*}
1\rightarrow \pi_{1}(\tilde{S})\rightarrow \pi_{1}(S) \xrightarrow{\phi} D \rightarrow 1
\end{equation*}

The image under $\phi$ of the generating set $\{a_{i}, b_{i}\}$ is a generating set for $D$. 

\begin{lem}
	\label{wellknownfact}
	Consider the presentation of $D$ given by
	\begin{equation}
		\langle \phi(a_{1}), \ldots, \phi(a_{g}), \phi(b_{1}), \ldots, \phi(b_{g})\mid r_{1}, \ldots, r_{s}\rangle
	\end{equation}
The subgroup $\pi_{1}(\tilde{S})$ of $\pi_{1}(S)$ is generated by the words $\{r_{i}'\}$ and their conjugates by elements of $\pi_{1}(S)$ that map to nontrivial elements of $D$. Here $r_{i}'$ is the word in the generating set of $\pi_{1}(S)$ obtained from $r_{i}$ by replacing each instance of $\phi(a_{j})$ by $a_{j}$ and $\phi(b_{j})$ by $b_{j}$.
\end{lem}
\begin{proof}
	This claim seems to be well-known. A proof is given in Lemma 4 of \cite{lifts}.
\end{proof}
\begin{cor}[Corollary of Lemma \ref{wellknownfact}]
\label{corollary}
Any finitely presented group $G$ can be embedded in the deck transformation group of a cover of a closed surface of genus $g$.
\end{cor}
\begin{proof}
Suppose $G:=\langle x_{1}, \ldots, x_{s}\mid r_{1}(x_{1}, \ldots, x_{s}), \ldots, r_{q}(x_{1}, \ldots, x_{s})\rangle$, where $s$ is less than $2g$. Let $\{g_{1}, \ldots, g_{2g}\}$ be a generating set of $\pi_{1}(S)$. Then by Lemma \ref{wellknownfact}, $G$ is the deck transformation group of the covering space corresponding to the subgroup of $\pi_{1}(S)$ normally generated by the words $\{r_{i}', g_{s+1}, \ldots, g_{2g}\}$. Here $r_{i}'$ is the word obtained from $r_{i}$ by replacing every instance of $x_{1}$ by the first generator, $g_1$, of $\pi_{1}(S)$, every instance of $x_{2}$ by the second generator of $\pi_{1}(S)$, etc. up to $x_{s}$. By the Freiheitssatz, $\{g_{1}, \ldots, g_{s}\}$ generates a free group, so all we are doing here is writing $G$ as a quotient of a free group in the usual way.\\

Now suppose $G$ has $n$ generators, where $n$ is larger than or equal to $2g$. It follows from the theory of HNN extensions that every countably generated group can be embedded in a group with two generators; a reference is Chapter 4 of \cite{LS}. A detail of this construction that will be needed later is that when embedding $G$ in a 2-generator group,  it is possible to choose a generator of $G$, and then construct the embedding in such a way that this generator is mapped to a generator of the 2-generator group. Embed $G$ in a two generator group $T$. Then $T$ is isomorphic to the deck transformation group of a cover of closed surface of genus $g$, as explained in the previous paragraph.
\end{proof}



\begin{cor}
	\label{corollary2}
	Let $c$ be simple curve represented by the conjugacy class of a word in $\pi_{1}(S)$, and $\tilde{S}\rightarrow S$ a regular cover of $S$. There are pairs of simple curves and covers for which there does not exist an algorithm to determine whether $c$ represents a word in the cover.
\end{cor}

\begin{proof}
To claim that this corollary follows from Corollary \ref{corollary} and the unsolvability of the word problem for finitely presented groups, it is necessary to show that the embedding of $G$ in the deck transformation group does not force us to identify all words that an algorithm is unable to distinguish from the identity with words that could not be simple in $\pi_{1}(S)$. \\
	
Let $G$ be a finitely presented group, with unsolvable word problem, for which there is no algorithm to decide whether or not $G$ is trivial. Fix a generator $g$ of $G$. In the proof of The proof of Corollary \ref{corollary} we saw that the embedding of $G$ in the deck transformation group $D$ can be done in such a way that $g$ is mapped to the image in $D$ of a simple curve $a_{1}$. If $G$ is the trivial group, then the image of $a_{1}$ is trivial in $D$, and hence by Lemma \ref{wellknownfact}, the simple curve $a_{1}$ represents a word in $\pi_{1}(\tilde{S})$. However, if $G$ is nontrivial, $g$ is a nontrivial element of $G$, and by Lemma \ref{wellknownfact}, $a_{1}$ does not represent a curve in the cover. This problem is therefore undecidable.
\end{proof}

\section{Covering spaces without simple curves}
\label{section2}
This section uses homological arguments and properties of normal subgroups of $[\pi_{1}(S), \pi_{1}(S)]$ to construct covering spaces without simple curves.\\

\textbf{Necessary conditions for curves to be simple. }A primitive homology class is an element of $H_{1}(S;\mathbb{Z})$ that is either trivial or can not be written as a multiple $k\in (\mathbb{Z}\setminus \{\pm 1\})$ of another element of $H_{1}(S;\mathbb{Z})$. On an orientable surface, a necessary condition for a curve to be simple is that it is a representative of a primitive homology class. For curves represented by words in $[\pi_{1}(S), \pi_{1}(S)]$, as proven in \cite{Notsimple}, a necessary condition for being simple is that the word is not contained in $[\pi_{1}(S),[\pi_{1}(S), \pi_{1}(S)]]$.\\

\textbf{Covers without simple curves. }An example of a covering space without simple curves is a covering with deck transformation group that can not be generated by any proper subset of $\{\phi(a_{1}),\ldots,\phi(a_{g}), \phi(b_{1}), \ldots, \phi(b_{g})\}$, and for which the words $\{r_{i}'\}$ in $\pi_{1}(S)$ from Lemma \ref{wellknownfact} are all in $[\pi_{1}(S),[\pi_{1}(S), \pi_{1}(S)]]$.\\

Another source of covers without simple curves  comes from Section 3.5 of \cite{lifts}. In \cite{lifts}, examples of covering spaces are given, for which connected components of pre-images of simple curves do not span the integer homology of the covering space. These covering spaces are constructed by iterating mod $m$ homology covers, as will now be explained briefly.\\

The mod $m$ homology cover is defined by 
\begin{equation}
\label{modm}
1\rightarrow \pi_{1}(\tilde{S})\rightarrow \pi_{1}(S) \xrightarrow{\phi} H_{1}(S;\mathbb{Z}_{m}) \rightarrow 1.
\end{equation}
Due to the fact that $\phi$ factors through the Hurewicz homomorphism, it follows from Lemma \ref{wellknownfact} that $[\pi_{1}(S),\pi_{1}(S)]\lhd \pi_{1}(\tilde{S})$. The commutator subgroup of $\pi_{1}(S)$ contains simple curves, so mod $m$ homology covers do have simple curves. However, the only simple curves in mod $m$ homology covers are contained in $[\pi_{1}(S),\pi_{1}(S)]$, because by Lemma \ref{wellknownfact}, none of the other elements of $\pi_{1}(\tilde{S})$ are contained in primitive homology classes in $H_{1}(S; \mathbb{Z})$. As shown in Lemma 3 of \cite{lifts}, none of the simple curves in $[\pi_{1}(S),\pi_{1}(S)]$ are contained in $[\pi_{1}(\tilde{S}),\pi_{1}(\tilde{S})]$, so any simple curves in a mod $m$ homology cover $\tilde{S}$ are nonseparating in $\tilde{S}$. Taking a mod $m$ homology cover $\tilde{\tilde{S}}$ of $\tilde{S}$ therefore kills off any simple curves that were in $\tilde{S}$. It follows that $\tilde{\tilde{S}}$ does not contain simple curves.\\

A characteristic covering space of $S$ is a cover invariant under the induced action of the mapping class group on $\pi_{1}(S)$.
The homomorphism $\phi$ from Equation \eqref{modm} is composed of two homomorphisms; the Hurewicz homomorphism, and the map from $H_{1}(S;\mathbb{Z})\rightarrow H_{1}(S;\mathbb{Z}_{2})$. Since both of these maps are invariant under the induced action of the mapping class group of $S$, it follows that mod $m$ homology covers are characteristic. Composing two mod $m$ homology covers to obtain $\tilde{\tilde{S}}$ therefore gives a regular cover of $S$.\\

\section{Proof of undecidability}
\label{section3}
This section starts with a discussion of properties of mod 2 homology covers. These will then be combined with the simpleness criteria from the previous section, and ideas coming from Rabin's construction \cite{Rabin}, to prove Theorem \ref{maintheorem}. 

\subsection{Mod 2 homology covers}
Some specific properties of mod 2 homology covers will now be discussed. 

\begin{lem}
	\label{m2}
	Let $\tilde{S}\rightarrow S$ be the mod 2 homology cover of $S$. Then $\pi_{1}(\tilde{S})$ is generated by $(2g)^{2}-1$ elements $\{g_{1}=c_{1}^{2}, \ldots, g_{n}=c_{n}^{2}\}$ and their conjugates by elements of $\pi_{1}(S)$ mapping to nontrivial elements of $D$. Here $c_{i}$ is a word representing a simple, nonseparating curve on $S$.
\end{lem}
\begin{proof}
We first claim that a complete set of relations for the deck transformation group $D$ is given by $\{d_{i}^{2}, i=1, \ldots, (2g)^{2}-1\}$, where $d_{i}$ is one of the $(2g)^{2}-1$ nontrivial elements of $D$. Note that this is special to the case $m=2$, as it is only for $m=2$ that every element of $H_{1}(S; \mathbb{Z}_{m})$ is self inverse, which implies the commutation relations:
\begin{equation*}
ab=(ab)^{-1}=b^{-1}a^{-1}=ba
\end{equation*}
$H_{1}(S; \mathbb{Z}_{2})$ is by definition the abelian group with $2g$ generators, for which every group element has order two. This proves the claim.\\
  
Using Lemma \ref{wellknownfact}, it remains to prove that each of the nontrivial elements of $H_{1}(S;\mathbb{Z}_{2})$ has a simple curve representative. Fix an element $h$ of $H_{1}(S;\mathbb{Z}_{2})$. Then $h$ has as representative a subset of the set of curves $\{a_{i}, b_{j}\}$. If this subset contains both $a_{i}$ and $b_{i}$ for some $i$, replace the curves $a_{i}$ and $b_{i}$ by the simple curve obtained by Dehn twisting $a_{i}$ around $b_{i}$. Therefore, $h$ is represented by a set of at most $g$ simple, pairwise disjoint curves. It is therefore possible to take a connected sum to obtain a simple curve representing $h$. A set of simple curves representing the elements of $H_{1}(S;\mathbb{Z}_{2})$ is denoted by $\{c_{1}, \ldots, c_{(n)}\}$.
\end{proof}

\subsection{Tower of covers}
We now have all the ingredients to prove the main theorem of this paper.
\begin{thm*}
There does not exist an algorithm that determines whether a normal subgroup $\pi_{1}(\tilde{S})\lhd \pi_{1}(S)$ contains a word representing a simple curve on $S$.	
\end{thm*}

From now on, for simplicity of notation, the same symbols will be used to denote elements of $\pi_{1}(S)$, and their images in the deck transformation group of the cover.
\begin{proof}
Start off by assuming $S$ is connected, orientable and with empty boundary. Denote by $S_{2}$ the mod 2 homology cover of $S$. In what follows, a tower of covers 
\begin{equation*}
\tilde{S}_{2}\rightarrow S_{2}'\rightarrow S_{2}\rightarrow S
\end{equation*}
will be constructed. We first construct the deck transformation group of a cover $S_{2}'$ of $S_{2}$. To describe $S_{2}'$, let $H=\langle h_{1}, h_{2}, \ldots h_{m}\mid R_{1},\ldots, R_{s}\rangle$ be a finitely presented group with unsolvable word problem. The group $H\ast \langle x\rangle$ will be embedded in a two generator group $U$, as will now be explained.\\
	
Recall the definition of $\{g_{i}\}$ from Lemma \ref{m2}. This subset of a generating set of $\pi_{1}(S_{2})$ will be mapped to generators (denoted by the same symbols) of the deck transformation group of the cover $S_{2}'$ of $S_{2}$. By the Freiheitssatz, the $(2g)^{2}-1$ elements $\{g_{i}\}$ generate a free subgroup of $\pi_{1}(S_{2})$. A cover of $S_{2}$ will now be constructed, with deck transformation group generated by the images of the group elements $\{g_{i}\}$.\\

Let $x:=g_{1}$, $a:=[g_{1},[g_{1},g_{2}]]$ and $b:=[g_{2},[g_{1},g_{2}]]$. The commutator subgroup of $\pi_{1}(S_{2})$ is a free group, so the subgroup $[\pi_{1}(S_{2}),[\pi_{1}(S_{2}), \pi_{1}(S_{2})]]$ of the commutator subgroup is also free, and $a$ and $b$ freely generate a free subgroup of this group. Similarly, the set 
\begin{equation*}
\{a, b^{-1}ab, b^{-2}ab^{2}, \ldots, b^{-(m-1)}ab^{m-1}\}
\end{equation*}
freely generates a free subgroup of $\langle a,b \rangle$.\\
	
Let $\psi:H\ast \langle x\rangle\rightarrow \langle a,b\rangle$ be defined by $\psi(h_{i})=b^{-(i-1)}ab^{i-1}$, and let $U$ be the group with presentation
\begin{equation*}
U=\langle g_{1}, g_{2} \mid R_{1}' , \ldots, R_{s}'\rangle
\end{equation*}
where $R_{i}'$ is the word in the generators $g_{1}$ and $g_{2}$ obtained by taking $R_{i}$ and replacing every instance of $h_j$ with the word in $\langle g_{1}, g_{2}\rangle$ given by $\psi(h_{j})$.\\

\begin{rem}
It is necessary here that $\langle g_{i} \rangle$ generates a free group. As a result, the above construction gives an embedding of $H$ in a larger group, in the usual way. The same is true for the embeddings of groups that follow.
\end{rem}

Now let 
\begin{equation*}
J:=\langle U, g_{3}, g_{4}\mid g_{3}g_{1}g_{3}^{-1}=g_{1}^{2},\ \  g_{4}^{-1}g_{2}g_{4}=g_{2}^{2}\rangle
\end{equation*}
\begin{equation*}
K_{0}:=\langle J, g_{5}\mid g_{5}^{-1}g_{3}g_{5}=g_{3}^{2},\ \  g_{5}^{-1}g_{4}g_{5}=g_{4}^{2}\rangle
\end{equation*}
\begin{equation*}
K_{1}:=\langle K_{0}, g_{6}\mid g_{6}^{-1}g_{5}g_{6}=g_{5}^{2}\rangle
\end{equation*}
\begin{equation*}\vdots\end{equation*}
\begin{equation*}
K_{(2g)^{2}-8}:=\langle K_{(2g)^{2}-7},\ \  g_{(2g)^{2}-4}\mid g_{(2g)^{2}-4}^{-1}g_{(2g)^{2}-5}g_{(2g)^{2}-4}=g_{(2g)^{2}-5}^{2}\rangle
	\end{equation*}
	\begin{equation*}
	Q:=\langle g_{(2g)^{2}-3}, \ \ g_{(2g)^{2}-2}, \ \ g_{(2g)^{2}-1}\mid g_{(2g)^{2}-2}^{-1}g_{(2g)^{2}-3}g_{(2g)^{2}-2}=g_{(2g)^{2}-3}^{2},\ \  g_{(2g)^{2}-1}^{-1}g_{(2g)^{2}-2}g_{(2g)^{2}-1}=g_{(2g)^{2}-2}^{2}\rangle
	\end{equation*}
	\begin{equation*}
	D_{w}:=\langle K_{(2g)^{2}-7}\ast Q\mid g_{(2g)^{2}-3}=g_{(2g)^{2}-4},\ \  g_{(2g)^{2}-1}=[w,x]\rangle
	\end{equation*}
where $w$ is a word in the generators $a$ and $b$ of $H$. Note that, given a presentation for $H$, it is possible to write down a presentation $D_{w}=\langle g_{1}, \ldots, g_{(2g)^{2}-1}\mid r_{1}, \ldots, r_{M}\rangle$ for $D_{w}$.\\

By Lemma \ref{wellknownfact}, $D_{w}$ is isomorphic to the deck transformation group of a cover $S_{2}'$ of $S_{2}$, as follows:
\begin{equation*}
\langle g_{1}, \ldots, g_{(2g)^{2}-1}, d_{i}(g_{1}), \ldots, d_{i}(g_{(2g)^{2}-1}) \mid r_{k}(g_{1}, \ldots, g_{(2g)^{2}-1}),  d_{j}(g_{1})=d_{j}(g_{2})=\ldots=d_{j}(g_{(2g)^{2}-1})=1\rangle
\end{equation*}
where $i=1, \ldots, (2g)^{2}-1$,  $j=2, \ldots,  (2g)^{2}-1$, $k=1, \ldots, M$, and $d_{i}$ is an element of the deck transformation group of the cover $S_{2}\rightarrow S$.\\
	
Claim: The following are equivalent:
\begin{enumerate}
\item{$w=1\in H$}
\item{$S_{2}' \rightarrow S_{2}$ is the trivial cover.}
\end{enumerate}
To show that 1 implies 2, recall that by construction, $[w,x]$ is only trivial in the deck transformation group when $w$ is the trivial word in $H$. When $w$ is the trivial word in $H$, this implies $g_{(2g)^{2}-1}$ is trivial, from which it then follows that $g_{(2g)^{2}-2}$ is trivial, and hence $g_{(2g)^{2}-3}, \ldots, g_{1}$ are all trivial. In this case $S_{2}^{1}$ is the trivial cover of $S_{2}$. \\
		
To show that 2 implies 1, when $w$ is not the trivial word in $H$, $D_{w}=\langle g_{1}, \ldots, g_{(2g)^{2}-1}\mid r_{1}, \ldots, r_{M}\rangle$ is a nontrivial group. This follows from normal form theorems for HNN extensions and free products with amalgamation arising from the work of \cite{HNN}. A reference is Section 2 of Chapter IV of \cite {LS}. These normal form theorems also imply that none of the generators $g_{1}, \ldots, g_{(2g)^{2}-1}$ are trivial in $D_{w}$. This proves the claim.\\
	
Now please note that $S_{2}'$ is a regular cover of $S_{2}$, but $\pi_{1}(S_{2}')$ is not a normal subgroup of $\pi_{1}(S)$. 
We now want to find a regular cover $\tilde{S}_{2}$ of $S$ factoring through $S_{2}'$. A necessary condition is that the deck transformation group of the cover $\tilde{S}_{2}\rightarrow S_{2}$ is invariant under the induced action of the deck transformation group of the cover $S_{2}\rightarrow S$. Without this condition, it is not possible to find a deck transformation group, hence the cover could not be regular. We therefore consider the covering spaces $S_{2,i}$ of $S_{2}$, where $S_{2,i}$ has deck transformation group $\tilde{D}_{w, i}$ given by
	\begin{equation*}
	\label{compose}
	\langle d_{j}(g_{1}), \ldots, d_{j}(g_{(2g)^{2}-1}) \mid r_{k}(d_{i}(g_{1}), \ldots, d_{i}(g_{(2g)^{2}-1})),  d_{j}(g_{1})=\ldots=d_{j}(g_{(2g)^{2}-1})=I\rangle
	\end{equation*}
where $i=0,1,\ldots, (2g)^{2}-1$, the index $j=0,1, \ldots, i-1, i+1, \ldots,  (2g)^{2}-1$, the index $k=1, \ldots, M$, and $d_{i}$ is an element of the deck transformation group of the cover $S_{2}\rightarrow S$, where $d_{0}=1$.\\

The group $\pi_{1}(\tilde{S}_{2})$ is defined to be the largest normal subgroup of $\pi_{1}(S)$ in the intersection of all the groups $\pi_{1}(S_{2,i})$. Denote by $N$ the intersection in $\pi_{1}(S)$ of the groups $\pi_{1}(S_{2,i})$. To show that $\pi_{1}(\tilde{S}_{2})$ is nontrivial, we first show that $N$ is nonempty. This is a consequence of Lemma \ref{wellknownfact}, as words of the form $r_{k}(d_{i}(g_{1}), \ldots, d_{i}(g_{(2g)^{2}-1}))$, for $i=0,1,\ldots, (2g)^{2}-1$ and $k=1, \ldots, M$ are in the intersection.\\

To show that $N$ contains a nontrivial normal subgroup, first note that $N$ is invariant under conjugation by elements $\{a_{i}, b_{j}\}$ that are mapped to generators of the deck transformation group. This is because conjugation by such elements merely permutes the groups whose intersection is $N$. Since $N$ is invariant under conjugation by the generators of $\pi_{1}(S)$, it follows that $N$ is normal in $\pi_{1}(S)$, and hence $N=\pi_{1}(\tilde{S}_{2})$.\\

Given a presentation for $D_{w}$, a presentation for $\pi_{1}(\tilde{S}_{2})$ can be obtained. The relations of this presentation consist of the relations of $D_{w}$, the conjugates by $\{a_{i}, b_{j}\}$ of the relations of $D_{w}$, and the image of the relation $\Pi_{i=1,\ldots, g}[a_{i},b_{i}]$ in $\pi_{1}(S)$.\\


By construction, $\pi_{1}(\tilde{S}_{2})$ is a nontrivial cover of $S_{2}$ if and only if $w$ is a nontrivial word in $H$. In the former case, $\tilde{S}_{2}\rightarrow S$ is the cover $S_{2}\rightarrow S$, which we have seen does contain simple curves. In the latter case, none of the words in $\pi_{1}(\tilde{S}_{2})$ represent simple curves on $S$, as these words are all in $[\pi_{1}(S), [\pi_{1}(S), \pi_{1}(S)]]$ or are representatives of nonprimitive homology classes. \\

The cover $\tilde{S}_{2}\rightarrow S$ therefore has simple curves if and only if $w$ is a trivial word in $H$. An algorithm deciding whether covers have simple curves or not could therefore be used to solve the word problem for finitely presented groups. For connected, orientable surfaces with empty boundary, the theorem follows by contradiction.	\\      

If $S$ has nonempty boundary, an identical construction proves the theorem, with one small difference being that a generating set for $\pi_{1}(S)$ would then consist of curves $\{a_{i}, b_{j}\}$ as before, as well as some curves freely homotopic to boundary curves. Also, by first factoring through a double cover, the above construction applies when $S$ is not orientable. Similarly, the proof can be generalised to surfaces that are not connected in the obvious way.
\end{proof}

\bibliography{bib}
\bibliographystyle{plain}
\end{document}